\documentclass{amsart}
\usepackage{amssymb}
\usepackage[initials]{amsrefs}
\usepackage{tikz-cd}
\usepackage[english]{babel}

\usepackage{graphicx}
\usepackage{latexsym}
\usepackage{amsmath,amssymb,amsthm,amsfonts,amscd}
\usepackage{amssymb, mathrsfs, enumerate}
\usepackage{mathtools}
\usepackage{colonequals}
\usepackage[dvipsnames]{xcolor}
\usepackage{comment}

\renewcommand{\epsilon}{\varepsilon}

\usepackage{mleftright}

\usepackage{tikz}	
\usepackage{pgfplots}	
\usetikzlibrary{arrows, angles, quotes, calc,through,backgrounds,matrix,decorations.pathmorphing,
	intersections, pgfplots.fillbetween, patterns}
 \usepgfplotslibrary{polar}
\pgfplotsset{compat=newest}
\usepackage{multirow}
\usepgfplotslibrary{fillbetween}
\pgfplotsset{compat=1.17}
\pgfdeclarelayer{ft}
\pgfdeclarelayer{bg}
\pgfsetlayers{bg,main,ft}
\usepackage{tikz-cd}

\usepackage{thmtools}

\usepackage[linkcolor=black, urlcolor=black ,citecolor=black, colorlinks=true, hypertexnames=false]{hyperref}
\usepackage[capitalize, nameinlink, noabbrev]{cleveref}

\newcommand*{\N}{\mathbb N}
\newcommand*{\Z}{\mathbb Z}
\newcommand*{\Q}{\mathbb Q}
\newcommand*{\F}{\mathbb F}

\newcommand{\progr}{\mathcal{A}}
\newcommand{\graph}{\mathcal{G}}
\newcommand{\edge}[1]{E(#1)} 

\newcommand{\vertex}[1]{V(#1)} 
\newcommand{\vertexemph}{V}
\newcommand{\tree}{\mathcal{T}}

\newcommand{\aut}{\operatorname{Aut}}
\newcommand{\bassserre}{\hat{T}}

\newcommand{\fin}{\operatorname{F}}
\newcommand{\FP}{\operatorname{FP}}

\newcommand{\Sym}{\operatorname{Sym}}
\newcommand{\universal}{\mathcal{U}}
\newcommand{\id}{\operatorname{id}}

\newcommand{\SL}{\operatorname{SL}}
\newcommand{\CW}{\operatorname{CW}}

\input{packagefactor}

\newtheorem{thm}{Theorem}[section]

\newtheorem{cor}[thm]{Corollary}

\newtheorem{prop}[thm]{Proposition}

\theoremstyle{definition}
\newtheorem{defi}[thm]{Definition}

\newtheorem{ex}[thm]{Example}

\newtheorem{rem}[thm]{Remark}

\title[Simple tdlc groups separated by finiteness properties]{Simple totally disconnected locally compact groups separated by finiteness properties}
\author{Laura Bonn}
\address{Karlsruhe Institute of Technology, Germany}
\curraddr{}
\email{}
\thanks{The first author received support from the Deutsche Forschungsgemeinschaft (DFG, German Research Foundation) – Project number 281869850.}

\author{Sebastian Giersbach}
\address{Justus Liebig University Giessen, Germany}
\curraddr{}
\email{sebastian.giersbach@uni-giessen.de}
\thanks{The second author received support from the DFG Heisenberg project WI 4079/6.}


\date{February 2025}

\begin{document}

\begin{abstract}
    We construct a sequence of simple non-discrete totally disconnected locally compact (tdlc) groups separated by finiteness properties; that is, for every positive integer $n$ there exists a simple non-discrete tdlc group that is of type $\fin_{n-1}$ but not of type $\fin_n$. This generalizes a result for discrete groups of Skipper--Witzel--Zaremsky. Furthermore, we construct a simple non-discrete tdlc group that is of type $\FP_2$ over $\Z$ but not compactly presented. Our examples arise as Smith universal groups $\universal(M, N)$ associated to permutation groups $M$ and $N$. We generalize a theorem of Haglund--Wise to tdlc groups and show that under mild conditions on $M$ and $N$ the finiteness properties of $\universal(M, N)$ reflect those of its local actions $M$ and $N$.
\end{abstract}

\maketitle

\section{Introduction}
In the class of discrete groups, the finiteness properties $\fin_n$ and $\FP_n$ generalize the notions of finite generation and finite presentability. A group is of type $\fin_1$ or $\FP_1$ if and only if it is finitely generated, and being finitely presented is equivalent to being of type $\fin_2$. The property $\fin_n$ implies $\FP_n$ over any commutative ring and for finitely presented groups the two notions coincide over $\Z$. However, Bestvina and Brady \cite{bestvina+brady} showed that this equivalence fails for groups that are not finitely presented. More recently, Skipper, Witzel and Zaremsky \cite{skipper+witzel+zaremsky} constructed the first sequence of simple groups separated by finiteness properties, that is, they constructed simple groups $G_n$ that are of type $\fin_{n-1}$ but not of type $\fin_n$.

Finiteness properties were first generalized to locally compact groups by Abels and Tiemeyer \cite{abels+tiemeyer}. In the setting of totally disconnected locally compact (tdlc) groups, Castellano and Corob Cook \cite{castellano+cook} subsequently gave a different but equivalent definition, which we will use throughout this paper. In the discrete case, these generalizations recover the classical finiteness properties.

In this paper, we construct further examples of non-discrete tdlc groups separated by finiteness properties. In particular, we extend the result of Skipper, Witzel and Zaremsky to the tdlc setting.

\begin{thm}\label{thm: simple non-discrete tdlc groups separated by finiteness properties}
    For every positive integer $n$ there exists a simple non-discrete tdlc group that is of type $\fin_{n-1}$ but not of type $\fin_n$.
\end{thm}

In \cite{castellano+weigel}*{Question 2}, Castellano and Weigel asked whether there exists a non-discrete tdlc group with trivial quasi-center that is of type $\FP_2$ over $\Q$ but is not compactly presented. We answer this question in the affirmative and, in fact, prove a stronger statement.

\begin{thm}\label{thm: simple non-discrete tdlc group of type FP_2 not compactly presented}
    There exists a simple non-discrete tdlc group that is of type $\FP_2$ over $\Z$ but is not compactly presented.
\end{thm}

Using different methods, Llosa Isenrich, Schesler and Wu \cite{llosaisenrich+schesler+wu} have recently proven the analogous result for discrete groups, i.e.\ they constructed a simple group that is of type $\FP_2$ over $\Z$ (in fact, of type $\FP_\infty$) but is not finitely presented.

Our examples arise from a construction introduced by Smith \cite{smith}. The Smith universal group $\universal(M, N)$ acts on a biregular tree with prescribed local actions $M$ and $N$. This construction generalizes the Burger--Mozes universal groups $\universal(F)$ \cite{burger+mozes}, which act on a regular tree with prescribed local action $F$. Under suitable conditions on $M$ and $N$, the resulting Smith group $\universal(M, N)$ is a simple non-discrete tdlc group. Using this construction, Smith obtained uncountably many pairwise non-isomorphic simple non-discrete tdlc groups that are all compactly generated.

In the main theorem we show that if the permutation groups $M$ and $N$ both have finitely many orbits and one of them is transitive, the finiteness properties of the Smith group $\universal(M, N)$ relate to those of its local actions $M$ and $N$. For the proof we use that under these conditions, the Smith group $\universal(M, N)$ splits as a finite graph of groups. For such groups, we prove the following.

\begin{thm}\label{main thm: finiteness properties of graph of groups}
    Let $G$ be a tdlc group which splits as a finite graph of groups $(\graph, \progr)$.

    \begin{enumerate}
        \item Suppose for each edge $e \in E(\graph)$ the edge group $\progr_e$ is of type $\FP_n$ over $\Z$. Then $G$ is of type $\FP_n$ over $\Z$ if and only if for each vertex $v \in V(\graph)$ the vertex group $\progr_v$ is of type $\FP_n$ over $\Z$.

        \item Suppose for each edge $e \in E(\graph)$ the edge group $\progr_e$ is of type $\fin_n$. Then $G$ is of type $\fin_n$ if and only if for each vertex $v \in V(\graph)$ the vertex group $\progr_v$ is of type $\fin_n$.
    \end{enumerate}
\end{thm}

This generalizes a result by Haglund and Wise \cite{haglund+wise} to tdlc groups. In terms of the Smith group $\universal(M, N)$ \cref{main thm: finiteness properties of graph of groups} yields the following.

\begin{cor}\label{main thm: smith finite}
    Let $M \leq \Sym(X)$, $N\leq \Sym(Y)$ be non-trivial closed subgroups with compact point stabilizers. 
    Suppose that both have only finitely many orbits and that one of them acts transitively. Let $G \coloneqq \universal(M, N)$. Then:
    \begin{enumerate}
        \item $G$ is of type $\FP_n$ over $\Z$ if and only if $M$ and $N$ are of type $\FP_n$ over $\Z$,
        \item $G$ is of type $\fin_n$ if and only if $M$ and $N$ are of type $\fin_n$.
    \end{enumerate}
\end{cor}

This paper is structured as follows. In \cref{sec: smith} we recall the definition of a Smith group $\universal(M, N)$ and its basic properties. In \cref{sec: finiteness} we prove \cref{main thm: finiteness properties of graph of groups}, and in \cref{sec: examples} we apply \cref{main thm: smith finite} to construct examples that prove \cref{thm: simple non-discrete tdlc groups separated by finiteness properties} and \cref{thm: simple non-discrete tdlc group of type FP_2 not compactly presented}.

\section{Smith groups}\label{sec: smith}
Let $X$ and $Y$ be two non-empty disjoint sets, each containing at least two elements. 
Let $M \leq \Sym(X)$ and $N\leq \Sym(Y)$ be two permutation groups. 
Consider a connected biregular tree $\tree$ with bipartition $\vertex{\tree}=\vertexemph_X \cup \vertexemph_Y$ of its vertices. Suppose that all vertices in $\vertexemph_X$ have degree $\vert X \vert$ and that all vertices in $\vertexemph_Y$ have degree $\vert Y \vert$. 
Then $\tree$ is a $(\vert X \vert, \vert Y \vert)$-biregular tree.

We take $\tree$ as an oriented tree, where each edge exists in both directions. For a vertex $v \in \vertex{\tree}$ we define $o(v) = \{e \in \edge{\tree} : o(e) = v\}$ as the set of all edges originating in $v$. Similarly we define $t(v) = \{e \in \edge{\tree} : t(e) = v\}$ as the set of all edges terminating in $v$. A function $l \colon \edge{\tree} \to X \cup Y$ is called a \textit{legal labeling} if
\begin{enumerate}
    \item for all $v \in \vertexemph_X$, $l\vert_{ o(v)} \colon o(v) \to X$ is a bijection,
    \item for all $w \in \vertexemph_Y$, $l\vert_{ o(w)} \colon o(w) \to Y$ is a bijection and
    \item for all $v \in \vertex{\tree}$, $l\vert_{ t(v)}$ is constant. 
\end{enumerate}

\begin{figure}[h]
   \centering
   \begin{tikzpicture}[scale=0.7, level distance=1.5cm,
            every node/.style={shape=circle,fill=black,circle,inner sep=1.5pt}
            ]
            \node(a) {}
                child [grow=0, line width=0.5mm, level distance=1cm , ForestGreen] {node(b) {} 
                    child {node(c) {}
                        child [grow=330, YellowOrange, level distance=1.5cm]{node(d){}} 
                        child [grow=30, cyan, level distance=1.5cm]{node(e){}} }}
                child [grow= 130, YellowOrange, level distance=1cm, line width=0.5mm] {node(f){}
                    child {node(g){}
                        child [grow=100, ForestGreen, level distance=1.5cm] {node(h){}}
                        child [grow=160, cyan, level distance=1.5cm]{node(i){}} }}
                child [grow=240, cyan, level distance=1cm, line width=0.5mm] {node(j){}
                    child {node(k){}
                        child[grow=210, ForestGreen, level distance=1.5cm] {node(l){}}
                        child[grow=270, YellowOrange, level distance=1.5cm] {node(m){}} }};
            \draw[every node/.style={font=\tiny},every edge/.style={draw=none}] 
                (a) edge["1"] (b)
                (b) edge["2"] (c)
                (c) edge["2"] (d)
                (c) edge["2"] (e)
                (a) edge["1"] (f)
                (f) edge["2"] (g)
                (g) edge["2"] (h)
                (g) edge["2"] (i)
                (a) edge["1"] (j)
                (j) edge["2"] (k)
                (k) edge["2"] (l)
                (k) edge["2"] (m);
        \end{tikzpicture}
   \caption{A legal labeling of the biregular tree $\tree_{2,3}$ on a ball of radius $3$, with $X=\{\text{1, 2}\}$ and $Y=\{\text{blue, green, orange}\}$.}
\end{figure}
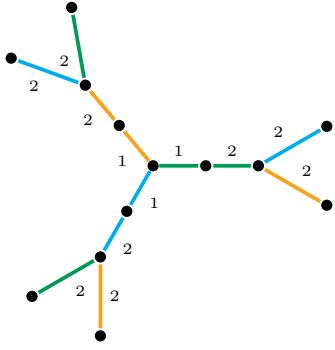

Denote the set of elements of $\aut(\tree)$ which fix $\vertexemph_X$ setwise by $\aut(\tree)_{\{\vertexemph_X\}}$. 
This set also fixes the set $\vertexemph_Y$, therefore $\aut(\tree)_{\{\vertexemph_X\}}=\aut(\tree)_{\{\vertexemph_Y\}}$.
Let $l$ be a legal labeling of $\tree$. For $v \in \vertexemph_X$ we define the bijection 
\begin{align*}
    l_v \colon o(v) \to X, \, e \mapsto l(e),
\end{align*}
and for $w \in \vertexemph_Y$
\begin{align*}
    l_w \colon o(w) \to Y, \, e \mapsto l(e).
\end{align*}
We get the maps into the permutation groups
\begin{align*}
    c_X \colon \aut(\tree)_{\{\vertexemph_X\}} \times \vertexemph_X \to \Sym(X), \, (g,v) \mapsto l_{g(v)} \circ g \circ l_v^{-1},\\
    c_Y \colon \aut(\tree)_{\{\vertexemph_Y\}}\times \vertexemph_Y \to \Sym(Y), \, (g,w) \mapsto l_{g(w)} \circ g \circ l_w^{-1}.
\end{align*}

\begin{defi}
    Let $X$ and $Y$ be two disjoint sets with at least two elements and let $M \leq \Sym(X)$ and $N\leq \Sym(Y)$ be two permutation groups. 
    Let $\tree$ be the $(\vert X \vert, \vert Y \vert)$-biregular tree and let $l$ be a legal labeling of $\tree$. 
    Then
    \begin{align*}
        \universal^l(M,N) \coloneqq \{ g \in \aut(\tree)_{\{\vertexemph_X\}} \mid \forall v \in \vertexemph_X \colon c_X(g,v) \in M \text{ and } \forall w \in \vertexemph_Y \colon c_Y(g,w) \in N\}
    \end{align*}
    is the \textit{Smith universal group} with local actions $M$ and $N$. 
\end{defi}
We directly see in this definition, that this construction is symmetric, hence we get $\universal^l(M,N)=\universal^l(N,M)$. By \cite{smith}*{Proposition 11}, for two different legal labelings $l$ and $l'$ the Smith groups $\universal^l(M, N)$ and $\universal^{l'}(M, N)$ are conjugate in $\aut(\tree)$. Hence, we will omit the legal labeling $l$ in our notation and simply write $\universal(M, N)$.

We will endow the Smith group $\universal(M, N)$ and, more generally, any permutation group with a topology. Let $Z$ be a non-empty set. The \textit{permutation topology}, also known as the \textit{topology of pointwise convergence}, on $\Sym(Z)$ has the pointwise stabilizers of finite subsets of $Z$ as a basis of identity neighborhoods. Under this topology any subgroup $G \le \Sym(Z)$ is a totally disconnected Hausdorff group. $G$ is open if and only if it contains the pointwise stabilizer of some finite subset of $Z$. $G$ is closed if and only if its pointwise stabilizer of some finite subset of $Z$ is closed. In particular, if the point stabilizer $G_z$ is finite for some $z \in Z$, then $G$ is closed.

As $\universal(M, N)$ is a subgroup of $\Sym(V(\tree))$, we can endow it with the permutation topology. The topological properties of $\universal(M, N)$ depend on those of $M$ and $N$:

\begin{prop}[\cite{smith}, Theorem 1 and Theorem 30]\label{prop: smith properties}
    Let $M \le \Sym(X)$ and $N \le \Sym(Y)$ be non-trivial permutation groups. Then the following hold.
    \begin{enumerate}
        \item If $M$ and $N$ are closed, then $\universal(M, N)$ is closed.
        \item Suppose $M$ and $N$ are closed. Then $\universal(M, N)$ is locally compact if and only if all edge stabilizers of $\universal(M,N)$ are compact if and only if all point stabilizers of $M$ and of $N$ are compact.
        \item $\universal(M, N)$ is discrete if and only if $M$ and $N$ act freely.
        \item Suppose $M$ and $N$ are generated by point stabilizers. Then $\universal(M, N)$ is simple if and only if $M$ or $N$ is transitive.
    \end{enumerate}
\end{prop}

\section{Finiteness properties}\label{sec: finiteness}

Before presenting the proofs of \cref{main thm: finiteness properties of graph of groups} and \cref{main thm: smith finite}, we briefly review the definition of finiteness properties of tdlc groups.
For further details we refer the reader to Castellano and Corob Cook \cite{castellano+cook}.

\begin{defi}
    Let $G$ be a tdlc group and $R$ be a commutative ring. 
    Then we call $G$ 
    \begin{itemize}
        \item of \emph{type $\fin_n$} if there exists a contractible proper discrete $G$-$\CW$ complex $X$, such that the $n$-skeleton $X^n$ is finite mod $G$.
        \item of \emph{type $\fin_{\infty}$} if $G$ is of type $\fin_n$ for every $n \in \N$.
        \item of \emph{type $\FP_n$ over $R$} if there exists a proper discrete resolution $P_* \to R$ of the trivial $R[G]$-module $R$, such that $P_0, \dots, P_n$ are finitely generated.
        \item of \emph{type $\FP_{\infty}$ over $R$} if $G$ is of type $\FP_n$ over $R$ for every $n \in \N$.
    \end{itemize}
\end{defi}

In contrast to finiteness properties of discrete groups, the resolution used to define the property $\FP_n$ is not projective. Therefore, we cannot use some statements about projective resolutions. However, many fundamental properties carry over to the tdlc setting.

\begin{prop}[\protect{\cite{castellano+cook}}]\label{prop: basic properties of finiteness properties}
    Let $G$ be a tdlc group and let $R$ be a commutative ring. 
    \begin{enumerate}
        \item $G$ is of type $\fin_1$ if and only if $G$ is compactly generated.
        \item $G$ is of type $\fin_2$ if and only if $G$ is compactly presented.
        \item If $G$ is of type $\fin_n$, then $G$ is of type $\FP_n$ over $R$. 
        \item If $G$ is compactly presented, then $G$ is of type $\fin_n$ if and only if $G$ is of type $\FP_n$ over $\Z$.
        \item Let $1 \to H \to G \to \factortext{G}{H} \to 1$ be a short exact sequence of tdlc groups. 
        Then the following statements hold over $R$:
        \begin{itemize}
            \item If $H$ is of type $\FP_n$ and $\factortext{G}{H}$ is of type $\FP_n$, then $G$ is of type $\FP_n$.
            \item If $H$ is of type $\FP_{n-1}$ and $G$ is of type $\FP_n$, then $\factortext{G}{H}$ is of type $\FP_n$.
        \end{itemize}
        The same implications hold with $\fin_n$ and $\fin_{n-1}$ in place of $\FP_n$ and $\FP_{n-1}$.
    \end{enumerate}
\end{prop}

As a first step, we show that the finiteness properties of the local actions $M\leq \Sym(X)$ and $N \leq \Sym(Y)$ relate to those of point stabilizers of $G$.

\begin{thm}[\cite{bonn}, Theorem 5.4.2]\label{thm: M N and point stabilizer}
    Let $M\leq \Sym(X)$ and $N \leq \Sym(Y)$ be non-trivial closed permutation groups with compact point stabilizers. 
    Let $G \coloneqq\universal(M,N)$ and let $R$ be a commutative ring.
    Then 
    \begin{enumerate}
        \item $M$ and $N$ are of type $\fin_n$ if and only if all point stabilizers of $G$ are of type $\fin_n$;
        \item $M$ and $N$ are of type $\FP_n$ over $R$ if and only if all point stabilizers of $G$ are of type $\FP_n$ over $R$. 
    \end{enumerate}
\end{thm}
\begin{proof}
    Since $M$ and $N$ are closed with compact point stabilizers, it follows from \cref{prop: smith properties} that $G$ is a tdlc group. Moreover, all edge stabilizers of $G$ are compact and therefore of type $\fin_\infty$ and of type $\FP_\infty$ over $R$.

    Let $v \in  \vertexemph_X$.
    Then the point stabilizer $G_v \vert_{B_1(v)} \le \Sym(B_1(v))$ of $v$ at the $1$-ball around $v$ is isomorphic to $M$ by \cite{smith}*{Lemma 13}. Hence, there exists a surjective group homomorphism $f \colon G_v \to M$ with $\ker(f)=\bigcap_{w \in B_1(v)} G_{(v,w)} = K$, see \cite{smith}*{Proposition 18}.
    As each edge stabilizer $G_{(v,w)}$ is compact, $K$ is also compact and therefore of type $\fin_\infty$ and of type $\FP_\infty$ over $R$. 
    The map $f$ induces the short exact sequence
    $$1 \to K \to G_v \to M \to 1.$$
    As $K $ is of type $\fin_\infty$ and of type $\FP_\infty$ over $R$, it follows from \cref{prop: basic properties of finiteness properties} that $G_v$ and $M$ have the same topological and homological finiteness properties for all $v \in \vertexemph_X$. The same argument applies for $G_w$ and $N$ for all $w \in V_Y$.
\end{proof}

Now we prepare our proof of \cref{main thm: finiteness properties of graph of groups}. A group that splits as a finite graph of groups can be described by iteratively taking amalgamated free products and HNN extensions. Thus, we begin by proving the base cases $n=1$ and $n=2$, that is compact generation and compact presentability, for amalgamated free products and HNN extensions. For further details about amalgamated free products and HNN extensions we refer the reader to Lyndon and Schupp \cite{lyndon+schupp}*{Chapter IV}.

\begin{prop}\label{prop: amalgam compactly generated}
    Let $G = A *_C B$ be an amalgamated free product of tdlc groups. If $G$ and $C$ are compactly generated, then so are $A$ and $B$.
\end{prop}
\begin{proof}
    We identify $C, A$ and $B$ with their images in $G$ under the canonical homomorphisms. As $C$ is a tdlc group, by van Dantzig's theorem it contains a compact open subgroup $K$. Since $C$ and $G$ are compactly generated, it follows from \cite{cornulier+delaharpe}*{Corollary 2.C.6.} that there exist finitely many elements $g_1, \ldots, g_m \in C, g_{m+1} \ldots, g_n \in G$ such that
    $$C = \left< K, g_1, \ldots, g_m \right> \text{ and } G = \left< K, g_1, \ldots, g_n \right>$$
    are compact generating sets. Using the normal form for amalgamated free products we can write
    $$g_i = a_{i, 1} b_{i, 1} a_{i, 2} b_{i, 2} \cdots a_{i, n_i} b_{i, n_i} \text{ where } a_{i, j} \in A, b_{i, j} \in B.$$ 
    Thus, we obtain a new compact generating set $G = \left< K, (a_{i, j}), (b_{i, j}) \right>$, adding $g_1, \ldots, g_m$ to both $(a_{i,j})$ and $(b_{i,j})$ if necessary.
    
    We will show that $A = \left< K, (a_{i, j}) \right>$ and hence that $A$ is compactly generated. Let $a \in A$. Since $a \in G$, it can be written as a product of elements in $K, (a_{i, j})$ and $(b_{i, j})$. Therefore, we can write $a = \alpha_1 \beta_1 \cdots \alpha_k \beta_k$, where each $\alpha_i$ is a word in $K$ and $(a_{i, j})$, and each $\beta_i$ is a word in $K$ and $(b_{i, j})$.
    Since $a \in A$, one reduced form of $a$ is $a$ itself. If $k > 1$, the product $\alpha_1 \beta_1 \cdots \alpha_k \beta_k$ is not a reduced form for $a$ because two reduced forms have the same length. In this case, some $\alpha_i$ or $\beta_i$ lies in $C$.
    Suppose $\alpha_i \in C$. Then, we can rewrite it as a product of elements in $K$ and $g_1, \ldots, g_m \in (b_{i, j})$ and $\beta_{i-1} \alpha_i \beta_i$ is a word in $K$ and $(b_{i, j})$. We replace $\beta_{i-1}$ with $\beta_{i-1} \alpha_i \beta_i$ and for $j \ge i$ we replace $\alpha_j$ with $\alpha_{j+1}$ and $\beta_j$ with $\beta_{j+1}$. If instead $\beta_i \in C$, we proceed similarly.
    
    By repeating this process, we eventually obtain a reduced form $a = \alpha_1 \beta_1$ or $a = \beta_1 \alpha_1$. As $a, \alpha_1 \in A$, we have $\beta_1 \in A \cap B = C$. We rewrite $\beta_1$ as a product of elements in $K$ and $g_1, \ldots, g_m \in (a_{i, j})$ and replace $\alpha_1$ with $\alpha_1 \beta_1$ (or $\beta_1 \alpha_1$). Thus, $a = \alpha_1 \in \left< K, (a_{i, j}) \right>$ and $A = \left< K, (a_{i, j}) \right>$ is compactly generated. An analogous argument shows that $B = \left< K, (b_{i, j}) \right>$ is compactly generated as well.
\end{proof}

\begin{prop}\label{prop: HNN compactly generated}
    Let $G = A *_C$ be an HNN extension of tdlc groups. If $G$ and $C$ are compactly generated, then so is $A$.
\end{prop}
\begin{proof}
    Let $\varphi \colon C \to C'$ be the isomorphism of the associated subgroups $C$ and $C'$. Let $K \le C$ be a compact open subgroup. There exist finitely many elements $g_1, \ldots, g_m \in C$ such that $C = \left< K, g_1, \ldots, g_m \right>$ is a compact generating set. Then $C' = \left< \varphi(K), \varphi(g_1), \ldots, \varphi(g_m) \right>$ is also a compact generating set. Let $K' := \varphi(K)$ and $g_{i+m} := \varphi(g_i)$. As $G$ is compactly generated, there exist finitely many elements $g_{2m+1}, \ldots, g_n \in G$ such that 
    $$
        G = \left< K, K', g_1, \ldots, g_n, t \right>
    $$
    is a compact generating set with $t$ being the stable letter of $G$.

    Using the normal form for HNN extensions we can write
    $$
        g_i = a_{i, 1} t^{\varepsilon_{i, 2}} a_{i, 2} \cdots t^{\varepsilon_{i, n_i}} a_{i, n_i} \text{ where } a_{i, j} \in A \text{ and } \varepsilon_{i,j} \in \{1, -1\}. 
    $$
    We obtain a compact generating set $G = \left< K, K', (a_{i, j}), t \right>$. Using the reduced form and applying similar reductions to \cref{prop: amalgam compactly generated} shows that $A = \left< K, K', (a_{i, j}) \right>$ is compactly generated.
\end{proof}

To prove \cref{main thm: finiteness properties of graph of groups} for compact presentability, we use a proposition that appears in \cite{castellano+cook} as an exercise. As we require a slightly stronger statement, we provide a full proof.

\begin{prop}[\cite{castellano+cook}, Proposition 3.7 (iii)]\label{prop: compactly presented colimit of compactly generated groups}
    Let $(G_n)_n$ be a sequence of tdlc groups with quotient maps $g_n \colon G_n \to G_{n+1}$. Let $G$ be the colimit of this sequence with quotient maps $\pi_n \colon G_n \to G$. If each $G_n$ is compactly generated and $G$ is compactly presented, then eventually $\pi_n \colon G_n \to G$ is an isomorphism.
\end{prop}

\begin{proof}
    By van Dantzig's theorem, there exists a compact open subgroup $C_0 \le G_0$. Since $G_0$ is compactly generated, it follows from \cite{cornulier+delaharpe}*{Corollary 2.C.6.} that there are finitely many elements $a_1, \ldots, a_k \in G_0$ such that $G_0 = \left< C_0, a_1, \ldots, a_k \right>$ is a compact generating set. For each $n$ define
    $$f_n := g_{n-1} \circ \ldots \circ g_0 \colon G_0 \to G_n \text{ and } C_n := f_n(C_0), C := \pi_0(C_0).$$
    As $f_n$ and $\pi_0$ are quotient maps and thus open and continuous, we get compact generating sets $G_n = \left< C_n, f_n(a_1), \ldots, f_n(a_k) \right>$ and $G = \left< C, \pi_0(a_1), \ldots, \pi_0(a_k) \right>$ with $C_n$ and $C$ being compact open subgroups.

    Let $K_n$ be the kernel of $f_n\vert_{C_0}$ and let $K$ be the kernel of $\pi_0\vert_{C_0}$. Then $K = \cup_n K_n$ is a countable union of compact subgroups. By the Baire category theorem, some $K_m$ is an open subgroup of $K$. As $K$ is compact, $K_m$ has finite index in $K$ and eventually $K_n = K$. For these $n$ the quotient map $\pi_n \vert_{C_n}$ is an isomorphism. Thus, eventually we may identify each $C_n$ with $C$ and the elements $f_n(a_i)$ with $b_i := \pi_0(a_i)$ and eventually each $G_n$ and $G$ have the same compact generating set $\left< C, b_1, \ldots, b_k \right>$.

    We now consider the general presentation $((\progr, \Lambda), \phi)$ of $G$ as in \cite{castellano+weigel}*{Proposition 5.10 (a)}. As $G$ and the fundamental group $\pi_1(\progr, \Lambda)$ are both compactly presented, we have a short exact sequence
    $$1 \rightarrow N \rightarrow \pi_1(\progr, \Lambda) \rightarrow G \rightarrow 1,$$
    where $N$ is compactly generated as a normal subgroup of $\pi_1(\progr, \Lambda)$ by \cite{cornulier+delaharpe}*{Proposition 8.A.10 (2)}. In fact, since $N$ is discrete by \cite{castellano+weigel}*{Proposition 5.10 (b)}, it is finitely generated as a normal subgroup and $N = \left< \left< r_1, \ldots, r_m \right> \right>$ for certain $r_1, \ldots, r_m \in N$. As $G$ is the colimit of the groups $(G_n)_n$, eventually the relations $r_1, \ldots, r_m$ hold in some $G_n$. For such $n$, the presentation $G = \left< C, b_1, \ldots, b_k \mid r_1, \ldots, r_m \right>$ also defines $G_n$. As $\pi_n \vert_{C_n} \colon C_n \to C$ is an isomorphism, it follows that $\pi_n \colon G_n \to G$ is an isomorphism as well.
\end{proof}

\begin{prop}\label{prop: amalgam compactly presented}
    Let $G = A *_C B$ be an amalgamated free product of tdlc groups. If $G$ and $C$ are compactly presented, then so are $A$ and $B$.
\end{prop}
\begin{proof}
    By \cref{prop: amalgam compactly generated} both $A$ and $B$ are compactly generated and there is a compact set $K \subseteq C$ and there are finitely many $a_i \in A, b_i \in B$ such that there are presentations
    $$C = \left< K \mid R_C \right>, A = \left< K, (a_i) \mid R_A \right> \text{ and } B = \left< K, (b_i) \mid R_B \right>$$
    where $R_A$ and $R_B$ are the sets containing all relations in $A$ and $B$. As $C$ is compactly presented, we can choose $R_C$ to only contain relations of bounded length. Define $R_A^n $ to be the set of relations in $A$ up to length $n$ and
    \begin{gather*}
        A_n := \left< K, (a_i) \mid R_A^n \right>, \\
        G_n := A_n *_C B = \left< K, (a_i), (b_i) \mid R_C, R_A^n, R_B \right>.
    \end{gather*}
    The group $G_n$ is well-defined for sufficiently large $n$ because the relations in $R_C$ have bounded length. We have the following commutative diagram.
    $$\begin{tikzcd}
    A_n \arrow[d, hook] \arrow[r, "\pi_n", two heads] & A \arrow[d, hook] \\
    G_n \arrow[r, "\varphi_n", two heads] & G
    \end{tikzcd}$$

    By \cref{prop: compactly presented colimit of compactly generated groups} eventually the quotient map $\varphi_n \colon G_n \to G$ is an isomorphism. Using commutativity, for such $n$ the map $\pi_n \colon A_n \to A$ is an isomorphism. Therefore, $A$ is compactly presented. A symmetric argument shows that $B$ is compactly presented as well.    
\end{proof}

\begin{prop}
    Let $G = A *_C$ be an HNN extension of tdlc groups. If $G$ and $C$ are compactly presented, then so is $A$.
\end{prop}
\begin{proof}
    Let $\varphi \colon C \to C'$ be the isomorphism of the associated subgroups. By \cref{prop: HNN compactly generated}, $A$ is compactly generated and there is a compact set $K \subseteq C$ with $K' := \varphi(K)$ and there are finitely many $a_i \in A$ such that there are presentations
    $$
        C = \left< K \mid R_C \right>, C' = \left< K' \mid R_{C'} \right> \text{ and } A = \left< K, K', (a_i) \mid R_A \right>
    $$
    where $R_A$ is the set containing all relations in $A$ and where $R_C$ and $R_{C'}$ contain relations of bounded length. Let $R_A^n$ be the set of relations in $A$ up to length $n$ and
    \begin{gather*}
        A_n := \left< K, K', (a_i) \mid R_A^n \right>, \\
        G_n := A_n *_C = \left< K, K', (a_i), t \mid R_C, R_{C'}, R_A^n, t^{-1}at = \varphi(a) : a \in K \right>.
    \end{gather*}
    As in the proof of \cref{prop: amalgam compactly presented}, eventually the quotient map $\varphi_n \colon G_n \to G$ is an isomorphism of topological groups. We get an isomorphism $\pi_n \colon A_n \to A$ and therefore, $A$ is compactly presented.
\end{proof}

Next, we prove \cref{main thm: finiteness properties of graph of groups} for $\FP_n$. For $n = 1$, Castellano \cite{castellano}*{Proposition 4.1} has proven the statement under the assumption of compact open edge groups. However, her proof idea can be adapted to the case $n > 1$.

\begin{thm}[\cite{bonn}, Theorem 5.4.6]\label{thm: graph of groups fin tdlc}
    Let $G$ be a tdlc group which splits as a finite graph of groups $(\graph, \progr)$. Suppose for each edge $e \in \edge{\graph}$ the edge group $\progr_e$ is of type $\FP_n$ over $\Z$. Then $G$ is of type $\FP_n$ over $\Z$ if and only if for each vertex $v \in \vertex{\graph}$ the vertex group $\progr_v$ is of type $\FP_n$ over $\Z$. 
\end{thm}
\begin{proof}
    As Brown's criterion also holds for tdlc groups, see \cite{castellano+cook}*{Theorem 4.7}, the finiteness properties of $\progr_v$ imply these for $G$.

    Let $\bassserre$ be the Bass-Serre tree of $(\graph,\progr)$.
    For $e \in \edge{\bassserre}$ and $v \in \vertex{\bassserre}$ let $H_e$ and $H_v$ denote the edge stabilizers and vertex stabilizers of $\bassserre$ respectively. 
    By construction of $G$ and the Bass-Serre tree, all stabilizers are open subgroups of $G$.
    Note that the edge and vertex stabilizers of the Bass-Serre tree  are given by edge and vertex groups of $(\graph,\progr)$ respectively. 
    Therefore, $H_e$ is of type  $\FP_n$ over $\Z$.

    Let $G$ be of type $\FP_n$ over $\Z$. Then $\Z$ is of type $\FP_n$ as a $\Z[G]$-module. 
    As $H_e$ is of type $\FP_n$ over $\Z$, $\Z$ is of type $\FP_n$ as a $\Z[H_e]$-module. 
    Since $H_e$ is an open subgroup, \cite{castellano+cook}*{Corollary 3.18} implies that $\Z[G] \otimes_{H_e} \Z=\Z[\factortext{G}{H_e}]$ is of type $\FP_n$ as a $\Z[G]$-module.
    The cellular chain complex of the Bass-Serre tree $\bassserre$ is given by the short exact sequence 
    \begin{align*}
        0 \to \bigoplus_{e \in R_e} \Z[\factor{G}{H_e}] \to \bigoplus_{v \in R_v} \Z[\factor{G}{H_v}] \to \Z \to 0
    \end{align*}
    of permutations modules over $\Z$, where $R_e$ and $R_v$ are representative systems of the edges respectively the vertices of the action of $\graph$ on $\bassserre$.
    By assumption, $\bigoplus_{e \in R_e}\Z[\factortext{G}{H_e}]$ and $\Z$ are both of type $\FP_n$ as $\Z[G]$-modules. Therefore, it follows by \cite{castellano+cook}*{Corollary 3.12} that $\bigoplus_{v \in R_v}\Z[\factortext{G}{H_v}]$ is of type $\FP_n$ as a $\Z[G]$-module. 
    It follows that $\Z[\factortext{G}{H_v}]$ is of type $\FP_n$ as a $\Z[G]$-module for all $v \in \vertex{\bassserre}$. 
    Since $H_v$ is an open subgroup and $\Z[\factortext{G}{H_v}]=\Z[G]\otimes_{H_v}\Z$, we can use \cite{castellano+cook}*{Corollary 3.18} again which implies that $\Z$ is of type $\FP_n$ as a $\Z[H_v]$-module. 
    Hence, $H_v$ is of type $\FP_n$ over $\Z$, so $\progr_v$ is of type $\FP_n$ over $\Z$ for all $v \in \vertex{\graph}$. 
\end{proof}

\begin{proof}[Proof of \cref{main thm: finiteness properties of graph of groups}]
    Part (1) is \cref{thm: graph of groups fin tdlc}.

    The if-implication of part (2) in the base cases $n = 1$ and $n = 2$ of compact generation and compact presentability follows from the definition of the fundamental group of a graph of groups. For compactly presented groups the properties $\fin_n$ and $\FP_n$ over $\Z$ are equivalent, so the if-implication for $n > 2$ follows from part (1).
    
    For the other implication recall that a group $G$ splitting as a finite graph of groups can be described by iteratively taking amalgamated free products and HNN extensions. Iterating \cref{prop: amalgam compactly generated} (resp. \cref{prop: amalgam compactly presented}) shows that all vertex groups are compactly generated (resp. compactly presented) if $G$ is. For $n > 2$ the only-if-implication follows again from part (1).
\end{proof}

\begin{proof}[Proof of \cref{main thm: smith finite}]
    By the assumption on the orbits of $M$ and $N$, it follows from \cite{smith}*{Lemma 22} that the quotient graph $\factorltext{\tree}{G}$ is a finite star graph. Hence, $G$ can be written as an iterated amalgamated free product with point stabilizers as factors and edge stabilizers as amalgamated subgroups. As the edge stabilizers are compact by assumption, they are of type $\fin_{\infty}$ and of type $\FP_{\infty}$. By \cref{main thm: finiteness properties of graph of groups}, $G$ is of type $\fin_n$ (resp. of type $\FP_n$ over $\Z$) if and only if each point stabilizer is of type $\fin_n$ (resp. of type $\FP_n$ over $\Z$). The statement now follows from \cref{thm: M N and point stabilizer}.
\end{proof}

\section{Examples}\label{sec: examples}
In this section we will construct non-discrete tdlc groups with certain finiteness properties. Taking a discrete group as the group $M$ and $\Sym(3)$ as the group $N$, the Smith group $\universal(M, N)$ will be a non-discrete tdlc group with the same finiteness properties as $M$. The easiest examples can be constructed if $M$ acts on itself by left multiplication. By taking other actions of $M$, the Smith group $\universal(M, N)$ can be arranged to be simple.

\begin{ex}
    A group $M$ acts on itself by left multiplication. Therefore, $M \le \Sym(M)$ is a transitive permutation group. As the action is free, $M$ has compact point stabilizers and is closed in $\Sym(M)$. Let $N = \Sym(3)$ be a permutation group (as a subgroup of $\Sym(3))$. Then $N$ is closed in $\Sym(3)$ and its action is transitive, not free and has compact point stabilizers that generate $N$. Since $N$ is a finite group, it is of type $\fin_{\infty}$ and of type $\FP_{\infty}$.

    By \cref{prop: smith properties}, the Smith group $\universal(M, N)$ is a non-discrete tdlc group. By \cref{main thm: smith finite}, it shares finiteness properties with $M$. In particular, we can choose $M$ to be a discrete group which is of type $\FP_2$ over $\Z$ but not finitely presented, for example a Bestvina--Brady group constructed in \cite{bestvina+brady}*{Examples 6.3}. Then $\universal(M, N)$ is a non-discrete tdlc group that is of type $\FP_2$ over $\Z$ but not compactly presented.
\end{ex}

We want to answer a stronger version of \cite{castellano+weigel}*{Question 2}. Furthermore, we will construct simple non-discrete tdlc groups with other finiteness properties. The simplicity criterion for Smith groups $\universal(M, N)$ needs both $M$ and $N$ to be generated by point stabilizers. As the action of $M$ on itself by left multiplication is free, one cannot expect $\universal(M, N)$ to be simple in this case. Instead we want $M$ to act on the coset space of certain subgroups.

\begin{ex}\label{ex: simple non-discrete tdlc group}
    Let $M$ be a group with a subgroup $Q$. Then $M$ acts on the  coset space $X = \factortext{M}{Q}$ by left multiplication. This action is faithful if and only if the intersection of all conjugates of $Q$ in $M$
    $$\bigcap_{m \in M} m Q m^{-1}$$
    is trivial. In that case $M \le \Sym(X)$ is a transitive permutation group with point stabilizers $M_{mQ} = m Q m^{-1}$. If $Q$ is finite, then $M$ has compact point stabilizers.
    
    Now assume that $M$ contains a finite subgroup $Q$ such the conjugates that $m Q m^{-1}$ intersect trivially and generate $M$. By \cref{prop: smith properties} and \cref{main thm: smith finite}, the Smith group $\universal(M, \Sym(3))$ is a simple non-discrete tdlc group which shares finiteness properties with $M$.
\end{ex}

We want to construct groups $M$ with finite subgroups $Q$ such that the conditions on the conjugates in \cref{ex: simple non-discrete tdlc group} are satisfied. For that, we will use Bestvina--Brady groups.

\begin{defi}
    Let $L$ be a finite flag complex with vertices $\{v_1, \ldots, v_n\}$. The \emph{right-angled Artin group} $A_L$ associated to $L$ is given by the presentation
    $$A_L = \left< v_1, \ldots, v_n \mid v_iv_j = v_jv_i \text{ for all edges } \{v_i, v_j\} \text{ in } L \right>.$$
    We have an epimorphism $\phi \colon A_L \to \Z$ which sends every generator $v_i \in A_L$ to $1 \in \Z$. The \emph{Bestvina--Brady group} $H_L$ associated to $L$ is the kernel of $\phi$.
\end{defi}

\begin{prop}[\cite{bestvina+brady}, Main Theorem]\label{prop: bestvina brady theorem}
    Let $L$ be a finite flag complex with Bestvina--Brady group $H_L$. Let $R$ be a non-trivial commutative ring.
    \begin{enumerate}
        \item $H_L$ is of type $\FP_{n+1}$ over $R$ if and only if $L$ is homologically $n$-connected over $R$.
        \item $H_L$ is finitely presented if and only if $L$ is simply connected.
    \end{enumerate}
\end{prop}

The idea is to take $M$ to be the semidirect product $H_L \rtimes \aut(L)$. This is well-defined as any flag automorphism of $L$ induces a group automorphism of $H_L$. We take $\aut(L)$ to be the finite subgroup of $M$. We need criteria to see if the conditions on the conjugates are met.

\begin{prop}\label{prop: semidirect product conjugates intersect trivially}
    Let $Q$ be a finite group acting on a torsion-free group $H$ by automorphisms. Denote by $M$ the corresponding semidirect product $H \rtimes Q$. If $Q$ acts faithfully on $H$, then the conjugates of $Q$ in $M$ intersect trivially.
\end{prop}
\begin{proof}
    For $(h, q) \in M$ we have
    \begin{align*}
        (h, q) \cdot Q \cdot (h, q)^{-1} &= \{ (h, q) \cdot (1, p) \cdot (h, q)^{-1} : p \in Q \} \\
        &= \{ (h \cdot ((qpq^{-1}). h^{-1}), qpq^{-1}) : p \in Q \} \\
        &= \{ (h \cdot (p'.h^{-1}), p') : p' \in Q \}.
    \end{align*}
    Let $(g, p) \in \bigcap_{h \in H} \{ (h \cdot (q.h^{-1}), q) : q \in Q \}$. Then $(g, p) = (h \cdot (p.h^{-1}), p)$ for every $h \in H$. In particular, $g = h \cdot (p.h^{-1})$ for every $h \in H$. We get $p.h = g^{-1}h$ for every $h \in H$, which means that $p$ corresponds to left multiplication by $g^{-1}$. As $Q$ is finite, $p$ has finite order. But $H$ is torsion-free, therefore $g = 1$ and $p$ corresponds to $\id_H$. Since $Q$ acts faithfully on $H$, $p = 1$ and the intersection of all conjugates of $Q$ is trivial.
\end{proof}

\begin{prop}\label{prop: bestvina brady generators}
    Let $L$ be a finite connected flag complex. Then the Bestvina--Brady group $H_L$ is finitely generated by the set $S = \{xy^{-1} : \{x, y\} \in E(L)\}$.
\end{prop}
\begin{proof}
    Since $L$ is connected, the set $S$ generates elements $xy^{-1}$ for any two vertices $x, y \in V(L)$: If $x, y$ are vertices with a path $x = x_0, x_1, \ldots, x_{n-1}, x_n = y$, then
    $$xy^{-1} = (x_0x_1^{-1})(x_1x_2^{-1}) \cdots (x_{n-2}x_{n-1}^{-1})(x_{n-1}x_n^{-1}) \in \left< S \right>.$$
    
    Let $h \in H_L$ and let $h = x_1^{b_1} \cdot \ldots \cdot x_n^{b_n}$ for $x_1, \ldots, x_n \in V(L)$ with $b_1 + \ldots + b_n = 0$. We start with a path from $x_1$ to $x_2$ consisting of $x_1 = y_1, y_2, \ldots, y_{n-1}, y_n = x_2$. Then as adjacent vertices commute in $H_L$ we have
    $$x_1^{b_1}x_2^{-b_1} = (y_1y_2^{-1})^{b_1} (y_2y_3^{-1})^{b_1} \cdots (y_{n-1}y_n^{-1})^{b_1} \in \left< S \right>.$$
    
    Let $a_n = \sum_{i=1}^n b_i$. Similarly we have $x_i^{a_i}x_{i+1}^{-a_i} \in \left< S \right>$ for every $1 \le i \le n-1$.
    Then
    $$h = x_1^{b_1} \cdot \ldots \cdot x_n^{b_n} = (x_1^{a_1}x_2^{-a_1})(x_2^{a_2}x_3^{-a_2}) \cdots (x_{n-1}^{a_{n-1}}x_n^{-a_{n-1}}) \in \left< S \right>$$
    and $H_L$ is generated by $S$.
\end{proof}

\begin{prop}\label{prop: bestvina brady conjugates generate everything}
    Let $L$ be a finite connected flag complex with automorphism group $Q$ such that every oriented edge $(x, y)$ lies in a triangle with vertices $x, y, z$ and there is a $q \in Q$ such that $q.(x, z) = (y, z)$. Then the conjugates of $Q$ in the semidirect product $M = H_L \rtimes Q$ generate the entire group $M$.
\end{prop}
\begin{proof}
    Recall that the conjugate of $Q$ by an element $(h, q)$ is
    $$\{ (h \cdot (p.h^{-1}), p) : p \in Q \}.$$
    
    By \cref{prop: bestvina brady generators}, it is sufficient to show that the conjugates of $Q$ generate the elements $(xy^{-1}, 1)$ for edges $\{x, y\} \in E(L)$. Let $\{x, y\}$ be an edge. By assumption there is a triangle with vertices $x, y, z$ and a $q \in Q$ such that $q.(x, z) = (y, z)$. In particular, $q.x = y$ and $q.z = z$. Then
    $$xz^{-1} \cdot (q. zx^{-1}) = xz^{-1} zy^{-1} = xy^{-1}$$
    and the element $(xy^{-1}, q) \in H_L \rtimes Q$ lies in a conjugate of $Q$. We get
    $$(xy^{-1}, q) \cdot (1, q^{-1}) = (xy^{-1} \cdot (q.1), qq^{-1}) = (xy^{-1}, 1)$$
    which shows that the conjugates of $Q$ generate $S$ and therefore $H_L$. Thus, they generate the entire group $M = H_L \rtimes Q$.
\end{proof}

\begin{thm}\label{thm: bestvina brady finiteness properties inheritance}
    Let $L$ be a finite connected flag complex such that every edge lies inside a triangle. Assume that the automorphism group $Q$ of $L$ acts transitively on oriented edges of $L$. Take $M = H_L \rtimes Q$ to be acting on $X = \factortext{M}{Q}$ and let $N = \Sym(3)$. Then $\universal(M, N)$ is a simple non-discrete tdlc group with the same finiteness properties as $H_L$.
\end{thm}
\begin{proof}
    As Bestvina--Brady groups are torsion-free, the conjugates of $Q$ in $M$ intersect trivially by \cref{prop: semidirect product conjugates intersect trivially}. By \cref{prop: bestvina brady conjugates generate everything}, the conjugates of $Q$ in $M$ generate the entire group $M$. Therefore, $\universal(M, N)$ is a simple non-discrete tdlc group which shares finiteness properties with $M$. As $H_L$ is a finite-index subgroup of $M$, they have the same finiteness properties.
\end{proof}

\begin{ex}
    Using Bestvina--Brady groups we can construct simple non-discrete tdlc groups that are of type $\fin_n$ but not of type $\fin_{n+1}$. For $n \ge 2 $ take $L$ to be the flag triangulation of the $n$-dimensional sphere $S^n$ given by the boundary of the $(n+1)$-dimensional cross-polytope. Let $Q$ be the automorphism group of $L$. For $n \ge 2$ the flag complex $L$ is simply connected. The sphere $S^n$ is $(n-1)$-connected but not $n$-connected. By \cref{prop: bestvina brady theorem}, the Bestvina--Brady group $H_L$ is of type $\fin_n$ but not of type $\fin_{n+1}$. As $n \ge 2$, every edge of $L$ lies inside a triangle. Furthermore, $Q$ acts transitively on oriented edges of $L$. By \cref{thm: bestvina brady finiteness properties inheritance}, $\universal(H_L \rtimes Q, \Sym(3))$ is a simple non-discrete tdlc group of type $\fin_n$ but not of type $\fin_{n+1}$.
\end{ex}

\begin{ex}\label{ex: simple non-discrete tdlc group FP_2 not compactly presented}
    For our last example, we want to construct a simple non-discrete tdlc group that is of type $\FP_2$ over $\Z$ but not compactly presented. For the Bestvina--Brady construction we need a finite connected flag complex $L$ satisfying the conditions of \cref{thm: bestvina brady finiteness properties inheritance} such that $L$ is not simply connected but is homologically $1$-connected over $\Z$.

    For the explicit construction of $L$ see the GAP code in the Appendix. There we also check the desired properties. For the theoretical background on this construction, we refer the reader to \cite{bridson+haefliger}*{Chapter III.$\mathcal{C}$}.
    
    Let $Q = \operatorname{PSL}_2(13) \times C_3 \times C_3$ which has three subgroups $V_1, V_2, V_3$ each isomorphic to $C_{13} \rtimes C_3$. We construct a graph $\Gamma$ as follows: The vertices of $\Gamma$ are the left cosets of $V_1, V_2, V_3$ in $Q$. The vertices $V_1$ and $xV_2$ are adjacent if there is a $y \in V_2$ such that $xy^{-1} \in V_1$. Analogously we get edges for $V_1, xV_3$ and $V_2, xV_3$. The left action by $Q$ on the cosets gives us the remaining adjacencies. Let $L$ be the triangle complex given by its $1$-skeleton $\Gamma$. Then $L$ is in fact a flag complex which is homologically $1$-connected over $\Z$ such that every edge lies inside a triangle. Furthermore, $Q$ acts on $L$ by automorphisms and transitively on oriented edges. Also, every closed loop in the link of any vertex has length at least $6$.

    The property not shown in the Appendix is that $L$ is not simply connected. This follows from a covering space argument: Every closed loop in the link of any vertex of $L$ has length at least $6$. As $L$ is a triangle complex, it follows from \cite{bridson+haefliger}*{Chapter II, Proposition 5.25} that $L$ has non-positive curvature. By the Cartan--Hadamard theorem the universal cover $\Tilde{L}$ is a $\operatorname{CAT}(0)$ space. By \cite{bridson+haefliger}*{Chapter II, Proposition 5.10}, $L$ has the geodesic extension property, so does its universal cover $\Tilde{L}$ which contains a geodesic line by \cite{bridson+haefliger}*{Chapter II, Proposition 5.8}. However, $L$ is bounded and so does not contain a geodesic line. Hence, $L$ and $\Tilde{L}$ are not homeomorphic and $L$ is not simply connected.

    As $L$ is homologically $1$-connected over $\Z$ but not simply connected, the Bestvina--Brady group $H_L$ is of type $\FP_2$ over $\Z$ but is not finitely presented by \cref{prop: bestvina brady theorem}. Let $M = H_L \rtimes \aut(L)$ be acting on $X = \factortext{M}{\aut(L)}$. Because every edge of $L$ lies inside a triangle and $\aut(L)$ acts transitively on oriented edges, by \cref{thm: bestvina brady finiteness properties inheritance} the Smith group $\universal(M, \Sym(3))$ is a simple non-discrete tdlc group which is of type $\FP_2$ over $\Z$ but is not compactly presented.
\end{ex}

\begin{rem}
    Bestvina--Brady groups give us one way to construct simple non-discrete tdlc groups with certain finiteness properties. However, one can find simple non-discrete tdlc groups separated by finiteness properties using other constructions.

    \begin{enumerate}
        \item Let $M$ be a group with a finite non-trivial subgroup $Q$. As both the intersection $\bigcap_{m \in M} m Q m^{-1}$ and the subgroup generated by the conjugates $mQm^{-1}$ are normal subgroups of $M$, the conditions on the conjugates of $Q$ are fulfilled if $M$ is simple. In this case, $\universal(M, \Sym(3))$ is a simple non-discrete tdlc group with the same finiteness properties as $M$. Taking $M$ to be one of the groups constructed in \cite{skipper+witzel+zaremsky}*{Theorem 7.1}, $\universal(M, \Sym(3))$ is of type $\fin_{n-1}$ but not of type $\fin_n$.

        \item Let $M = \SL_n(\F_q[t, t^{-1}])$ for coprime $n \ge 3$ and $q-1$. Consider its finite subgroup $Q = \SL_n(\F_q)$. As $n$ and $q-1$ are coprime, the intersection $\bigcap_{m \in M} m Q m^{-1}$ is trivial. By the fact that the special linear group over a Euclidean ring is generated by elementary matrices, see \cite{hahn+omeara}*{1.2.11}, it follows that the conjugates $m Q m^{-1}$ generate the entire group $M$. We leave the details as an exercise for the reader. By \cite{witzel}*{Theorem 3.35}, $M$ is of type $\fin_{2n-3}$ but not of type $\fin_{2n-2}$. Thus, $\universal(M, \Sym(3))$ is a simple non-discrete tdlc group with the same finiteness properties.

    \end{enumerate}
\end{rem}

\section{Acknowledgments}
We thank Thomas Titz Mite for constructing the complex used in \cref{ex: simple non-discrete tdlc group FP_2 not compactly presented}. The second author thanks his advisor Stefan Witzel for many fruitful discussions and for his guidance.

\section{Appendix}

\begin{verbatim}
LoadPackage("GRAPE");
LoadPackage("HAP");
LoadPackage("SIMPCOMP");

F := FreeGroup(["a1", "a2", "b1", "b2", "c1", "c2"]);
AssignGeneratorVariables(F);

local_rels_A := [ a1^3, a2^13, 
    a2*a1^-1*a2^-1*a1*a2^2, (a2*a1^-1)^3, (a2^-1*a1^-1)^3 ];
local_rels_B := [ b1^3, b2^13, 
    b2*b1^-1*b2^-1*b1*b2^2, (b2*b1^-1)^3, (b2^-1*b1^-1)^3 ];
local_rels_C := [ c1^3, c2^13, 
    c2*c1^-1*c2^-1*c1*c2^2, (c2*c1^-1)^3, (c2^-1*c1^-1)^3 ];
triangle_rels :=
    [ a1*b1*c1, a2*b2*c2, a1^2*a2^2*b1^2*b2^2*c1^2*c2^2,
    a1^2*a2^5*b1^2*b2^5*c1^2*c2^5, a1*a2^9*b1*b2^9*c1*c2^9 ];
pi := F/Concatenation(local_rels_A, 
local_rels_B, local_rels_C, triangle_rels);

d0 := DerivedSubgroup(pi);
phi0 := IsomorphismFpGroup(d0);
d1 := Image(phi0);
phi1 := IsomorphismSimplifiedFpGroup(d1);
d2 := Image(phi1);
phi2 := IsomorphismSimplifiedFpGroup(d2);
d3 := Image(phi2);
psi := GQuotients(d3, PSL(2,13))[1];
k_in_d3 := Kernel(psi);
k := PreImage(phi0*phi1*phi2, k_in_d3);
theta := NaturalHomomorphismByNormalSubgroup(pi,k);
q := Image(theta);
# StructureDescription(q) -> PSL(2, 13) x C_3 x C_3

V1 := Group([q.1,q.2]);
V2 := Group([q.3,q.4]);
V3 := Group([q.5,q.6]);
# StructureDescription(V1) -> C_13 : C_3

ConstructQuotientSkeleton := function()
    local cosets1, cosets2, cosets3, vertices, relation;
    cosets1 := RightCosets(q, V1);
    cosets2 := RightCosets(q, V2);
    cosets3 := RightCosets(q, V3);
    vertices := Concatenation(cosets1, cosets2, cosets3);
    relation := function(v,w)
        local x,y,u1,u2;
        x := Representative(v);
        y := Representative(w);
        u1 := v*x^-1;
        u2 := w*y^-1;
        if not u1=u2 then
            return ForAny(w, z->z*x^-1 in u1);
        else 
            return false;
        fi; 
    end;
    return Graph(q,vertices, OnRight, relation, true);
end;
gamma := ConstructQuotientSkeleton();

# The following triangles are given by CompleteSubgraphs(gamma, 3),
# By CompleteSubgraphs(gamma) these are maximal complete subgraphs.
# Therefore we get a flag complex.
reps_triangles :=
    [ [ 1, 253, 505 ], [ 76, 253, 505 ], [ 123, 253, 505 ],
    [ 163, 253, 505 ], [ 166, 253, 505 ], [ 189, 253, 505 ] ];
triangles := Concatenation(Orbits(gamma.group,
    reps_triangles, OnSets));
complex := SimplicialComplex(triangles);
complex2 := SCFromFacets(triangles);
# gamma, complex and complex2 describe the same triangle complex.

# Check Aut(gamma) acts transitively on oriented edges.
Size(Orbits(AutGroupGraph(gamma), DirectedEdges(gamma), OnTuples));

# Check complex is homologically 1-connected.
SCIsConnected(complex2);
h1 := Homology(complex,1);

# Check every closed loop in the link of
# every vertex has length at least 6.
Adj := Adjacency(gamma, 1);
LinkGraph := InducedSubgraph(gamma, Adj);
Girth(LinkGraph);
\end{verbatim}

  \bibliography{lit}
  \bibliographystyle{plain}

\end{document}